\documentclass[12pt,reqno]{amsart}

\usepackage{amsmath,amsthm,amssymb,comment,fullpage}
\usepackage{times}
\usepackage[T1]{fontenc}
\usepackage{mathrsfs}
\usepackage{latexsym}
\usepackage[dvips]{graphics}
\usepackage{epsfig}
\usepackage{amsmath,amsfonts,amsthm,amssymb,amscd}
\input amssym.def
\input amssym.tex
\usepackage{color}
\usepackage{hyperref}
\usepackage{url}
\usepackage{breakurl}
\usepackage{comment}
\newcommand{\bburl}[1]{\textcolor{blue}{\url{#1}}}

\numberwithin{equation}{section}

\newtheorem{thm}{Theorem}[section]

\theoremstyle{plain}

\newcommand\be{\begin{equation}}
\newcommand\ee{\end{equation}}
\newcommand\bea{\begin{eqnarray}}
\newcommand\eea{\end{eqnarray}}
\newcommand\bi{\begin{itemize}}
\newcommand\ei{\end{itemize}}
\newcommand\ben{\begin{enumerate}}
\newcommand\een{\end{enumerate}}
\newcommand\bc{\begin{center}}
\newcommand\ec{\end{center}}
\newcommand\ba{\begin{array}}
\newcommand\ea{\end{array}}

\newcommand\frakfamily{\usefont{U}{yfrak}{m}{n}}
\DeclareTextFontCommand{\textfrak}{\frakfamily}

\newtheorem{rek}[thm]{Remark}

\newcommand{\ncr}[2]{{#1 \choose #2}}

\newcommand{\hr}[1]{\href{#1}{\url{#1}}}

\newcommand{\mandm}{M\&M}
\newcommand{\mandms}{M\&M'S}

\title{The M\&M Game: From Morsels to Modern Mathematics}

\author{Ivan Badinskki}
\email{\textcolor{blue}{\href{mailto:ivan.badinski@gmail.com}{ivan.badinski@gmail.com}}}
\address{Department of Economics, MIT, Cambridge, MA 02139}

\author{Christopher Huffaker}
\email{\textcolor{blue}{\href{mailto:ckh1@williams.edu}{ckh1@williams.edu}}}
\address{Department of Mathematics and Statistics, Williams College, Williamstown, MA 01267}

\author{Nathan McCue}
\email{\textcolor{blue}{\href{mailto:nrm2@williams.edu}{nrm2@williams.edu}}}
\address{Department of Mathematics and Statistics, Williams College, Williamstown, MA 01267}

\author{Cameron N. Miller}

\author{Kayla S. Miller}

\author{Steven J. Miller}
\email{\textcolor{blue}{\href{mailto:sjm1@williams.edu, Steven.Miller.MC.96@aya.yale.edu}{sjm1@williams.edu,Steven.Miller.MC.96@aya.yale.edu} }}
\address{Department of Mathematics and Statistics, Williams College, Williamstown, MA 01267}

\author{Michael Stone}
\email{\textcolor{blue}{\href{mailto:ms14@williams.edu}{ms14@williams.edu}}}
\address{Department of Mathematics and Statistics, Williams College, Williamstown, MA 01267}

\thanks{This work was partially supported by NSF Grant DMS1265673. The fifth named author thanks Donald Cameron for the opportunity to talk on this problem at the 110\textsuperscript{th} meeting of the Association of Teachers of Mathematics in Massachusetts. We thank Stan Wagon for suggesting and sharing his solution to the probability the first shooter wins the hoops game when both probabilities are independent uniform random variables, and Frank Morgan for comments on an earlier draft.}

\subjclass[2010]{TBD (primary) TBD (secondary)}


\date{\today}

\begin{document}

\maketitle

\begin{abstract} To an adult, it's obvious that the day of someone's death is not precisely determined by the day of birth, but it's a very different story for a child. When the third named author was four years old he asked  his father, the fifth named author: If two people are born on the same day, do they die on the same day? While this could easily be demonstrated through murder, such a proof would greatly diminish the possibility of teaching additional lessons, and thus a different approach was taken. With the help of the fourth named author they invented what we'll call \emph{the M\&M Game}: Given $k$ people, each simultaneously flips a fair coin, with each eating an M\&M on a head and not eating on a tail. The process then continues until all \mandms\  are consumed, and two people are deemed to die at the same time if they run out of \mandms\  together\footnote{Is one really living without \mandms?}. This led to a great concrete demonstration of randomness appropriate for little kids; it also led to a host of math problems which have been used in probability classes and math competitions. There are many ways to determine the probability of a tie, which allow us in this article to use this problem as a springboard to a lot of great mathematics, including memoryless process, combinatorics, statistical inference, graph theory, and hypergeometric functions.
\end{abstract}

\tableofcontents

\section{The Origins of The Game}

The M\&M Game began as a simple question asked by Steven Miller's curious four year-old son Cam: If two people are born on the same day, do they die on the same day? Of course, needing a way to explain randomness to children (two year old Kayla was there as well), the three Millers took the most logical next step and used \mandms\  to give the answer - with a more fun question! This led to what we now call the \mandm\ Game (see Figure \ref{fig:MandMgamepics} for an illustration):\begin{quote}  \emph{You and some friends start with some number of \mandms. Everyone  flips a fair coin at the same time; if you get a head you eat an \mandm; if you get a tail you don't.  You continue tossing coins together until no one has any \mandms\ left, and whoever is the last person with an \mandm\ lives longest and `wins'.} \end{quote}

\begin{figure}[h]
\begin{center}
\scalebox{.537}{\includegraphics{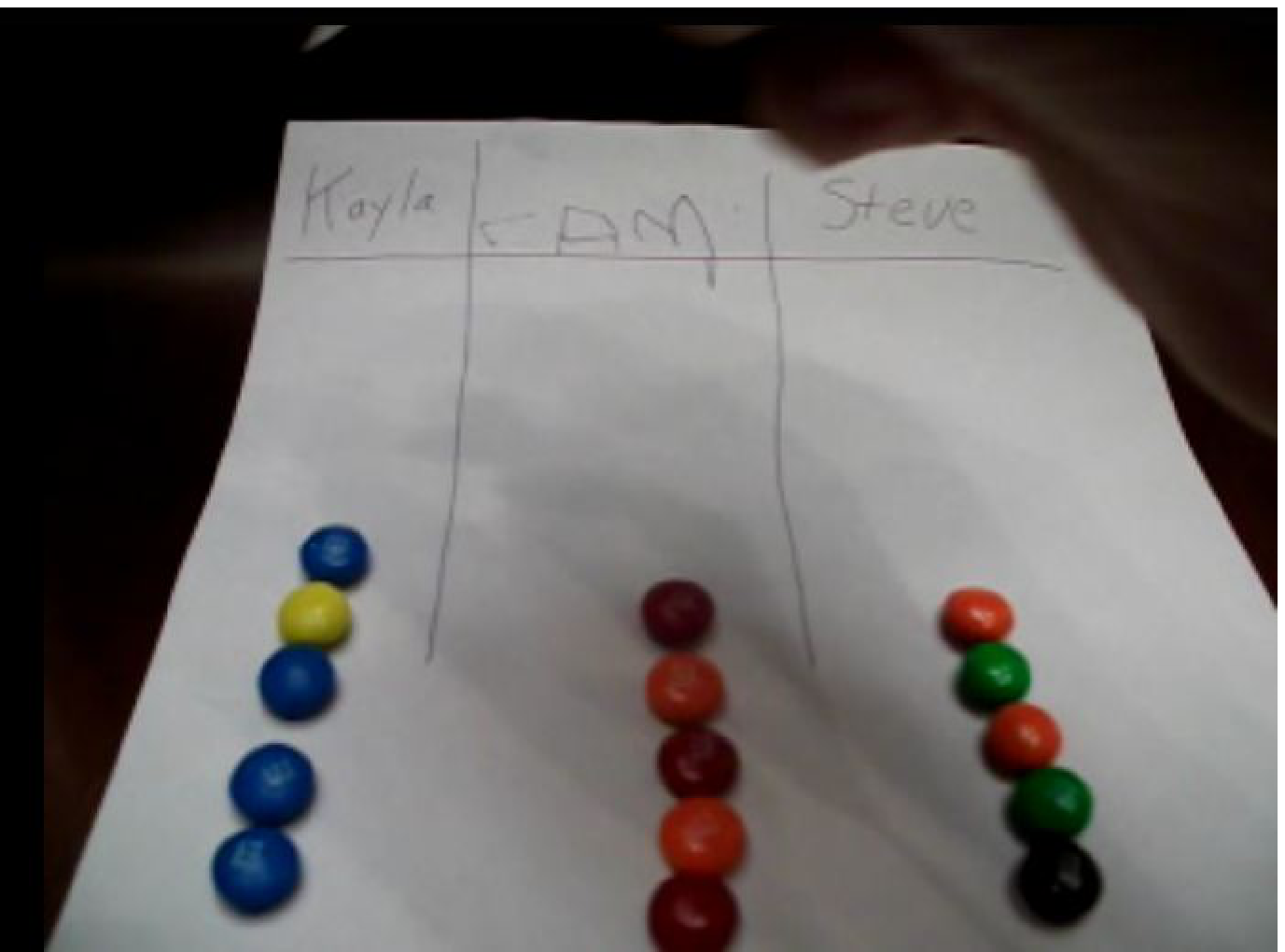}}\ \scalebox{.5}{\includegraphics{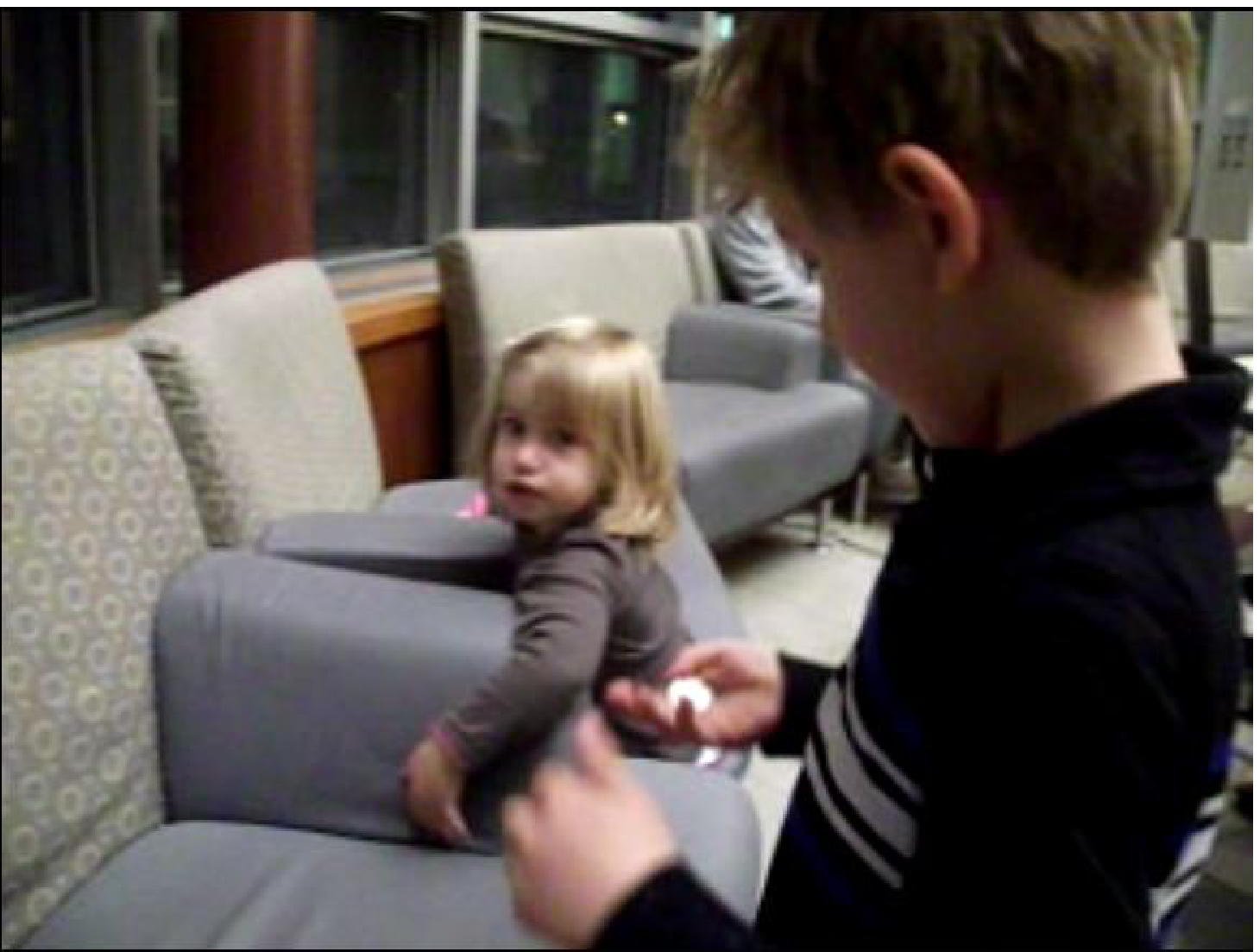}}
\caption{\label{fig:MandMgamepics} The first \mandm\ Game; for young players there is an additional complication in that it matters which colors you have, and the order you place them down.}
\end{center}\end{figure}

We can reformulate Cam's  question on randomness to: If everyone starts with the same number of \mandms, what is the chance everyone eats their last M\&M at the same time? In the arguments below we'll concentrate on two people playing with $c$ (for Cam) and $k$ (for Kayla) \mandms, though we encourage you to extend to the case of more people playing, possibly with a biased coin. As we will see in the following analysis, probability games like this one are a great way to see useful but complicated mathematical processes. In the course of our investigations we'll see some nice results in combinatorics and graph theory, and see applications of memoryless processes, statistical inference and hypergeometric functions. Such consequences are typical of good problems: in addition to being interesting, they serve as an excellent springboard to good concepts.

Recalling that the binomial coefficient $\ncr{n}{r} = \frac{n!}{r!(n-r)!}$ denotes the number of ways to choose $r$ objects from $n$ when order doesn't matter, we can compute the probability $P(k,k)$ of a tie when two people start with $k$ \mandms. If we let $P_n(k,k)$ denote the probability that the game ends in a tie with both people starting with $k$  \mandms\ after \emph{exactly} $n$ moves, then $$P(k,k) \ = \ \sum_{n=k}^\infty P_n(k,k);$$ note that we are starting the sum at $k$ as it is impossible all the \mandms\ are eaten in fewer than $k$ moves (we could start the sum at zero, but since $P_n(k,k) = 0$ for $n < k$ there is no need).

We claim that $$P_n(k,k) \ = \  \ncr{n-1}{k-1} \left(\frac12\right)^n \ncr{n-1}{k-1} \left(\frac12\right)^n.$$ This formula follows from the following observation: if the game ends in a tie after $n$ tosses, then each person has \emph{exactly} $k-1$ heads in their first $n-1$ tosses. As we have a fair coin, each string of heads and tails of length $n$ for a player has probability $(1/2)^n$. The number of strings for each person where the first $n-1$ tosses have \emph{exactly} $k-1$  heads, and the $n$\textsuperscript{th} toss is a head (we need this as otherwise we do not have each person eating their final \mandm\ on the $n$\textsuperscript{th} move) is $\ncr{n-1}{k-1} \ncr{1}{1}$. The $\ncr{1}{1}$ reflects the fact that the last toss must be a head; as this is just 1 it is common to omit that factor. As there are two players, the probability that each has their $k$\textsuperscript{th} head after the $n$\textsuperscript{th} toss is the product, proving the formula.

We have thus shown the following.

\begin{thm}\label{thm:maininfiniteexpansion} The probability the \mandm\ Game ends in a tie with two people using fair coins and starting with $k$ \mandms\ is \be\label{eq:probtiek} P(k,k) \ = \ \sum_{n=k}^\infty \ncr{n-1}{k-1} \left(\frac12\right)^n \ncr{n-1}{k-1} \left(\frac12\right)^n \ = \ \sum_{n=k}^\infty \ncr{n-1}{k-1}^2 \frac{1}{2^{2n}}.\ee
\end{thm}

While the above formula solves the problem, it is unenlightening and difficult to work with. The first difficulty is that it involves an infinite sum over $n$.\footnote{In general we need to be careful and make sure any infinite sum converges; while we are safe here as we are summing probabilities, we can elementarily prove convergence. Note $\ncr{n-1}{k-1} \le n^{k-1}/k!$, and thus the sum is bounded by $ k!^{-2} \sum_{n \ge k} n^{2k-2} / 2^{2n}$; as the polynomial $n^{2k-2}$ grows significantly slower than the exponential factor $2^{2n}$, the sum rapidly converges. } Second, it is very hard to sniff out the $k$-dependence: if we double $k$, what does that do to the probability of a tie? It is highly desirable to have exact, closed form solutions so we can not only quickly compute the answer for given values of the parameter, but also get a sense of how the answer changes as we vary those inputs. In the sections below we'll look at many different approaches to this problem, most of them trying to convert the infinite sum to a more tractable finite problem.

\section{The Basketball Problem, Memoryless Processes and the Geometric Series Formula}

\subsection{A Basketball Game}

It turns out that we can easily convert the infinite \mandm\ Game sum, equation \eqref{eq:probtiek}, into a finite sum using a powerful observation: we have a \emph{Memoryless Process}. Briefly, what this  means is that the behavior of the system only depends on the values of the parameters at a given moment in time, and not on how we got there.

There are many examples where all that matters is the configuration, not the path taken to reach it. For example, imagine a baseball game. If the lead-off hitter singles or walks, the net effect is to have a runner on first and the two results are the same.\footnote{For the baseball purist, there could be a very slight difference as a single breaks up a no-hit attempt, and if the next 26 batters are retired the pitcher might perform differently with a no-hitter on the line!} For another example,  consider a game of Tic-Tac-Toe; what matters are where the X's and O's are on the board, not the order they are placed. While chess at first might seem like a perfect example, it fails as many people play that if there is ever a configuration repeated three times in the game then the game is declared a draw; thus in chess we need to know \emph{how} we reached our state, and not just what state we are in.

Before delving into the reduction of the \mandm\ Game into a finite problem, we'll look at a related problem that's a little simpler but illustrates the same point. Moreover, we can easily extract from this problem the famous geometric series formula!

Imagine two of the greatest basketball players of all time, Larry Bird of the Boston Celtics and Magic Johnson of the Los Angeles Lakers\footnote{The players chosen reflect the childhood experiences of the eldest author.} are playing a basketball game. Instead of the intense competition which characterized the matches between their teams (see Figure \ref{fig:BirdMagic}) they instead play a one-on-one game of hoops as follows. \begin{quote} \emph{In this contest, Bird and Magic alternate shooting free throws, with Bird going first. Assume Bird always makes a basket with probability $p_L$, while Magic always gets a basket with probability $p_M$. If the probability Bird wins is $x_B$, what is $x_B$?} \end{quote}

\begin{figure}[h]
\begin{center}
\scalebox{.85}{\includegraphics{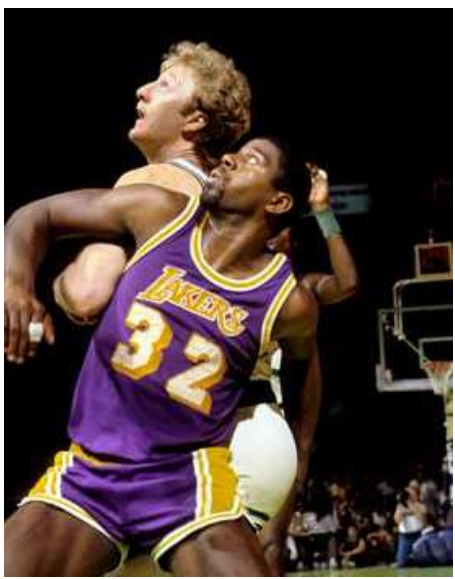}}
\caption{\label{fig:BirdMagic} Larry Bird and Magic Johnson, Game 2 of the 1985 NBA Finals (Boston, MA). Image from Steve Lipofsky from Wikipedia Commons (\bburl{http://www.basketballphoto.com/NBA_Basketball_Photographs.htm}).}
\end{center}\end{figure}

Note that this is almost a simplified M\&M Game: there is only one M\&M, but the players take turns flipping their coins. We'll see, however, that it is straightforward to modify the solution.

\subsection{Solution from the Geometric Series Formula}

The standard way to solve this problem uses a geometric series. Similar to the analysis in the introduction, the probability that Bird wins is the sum of the probabilities that Bird wins on his $n$\textsuperscript{th} shot. We'll see in the analysis below that it's algebraically convenient to define $r := (1-p_B)(1-p_M)$, which is the probability they both miss.\footnote{A quick word on notation. We use $p$ to denote probability, and put subscripts $B$ and $M$ so we can easily determine if we're talking about Bird or Magic; we use the letter $r$ for ratio, which will make sense when we see the geometric series with ratio $r$ emerge shortly. There is enormous value in good notation -- we can get a better understanding of what is going on simply by glancing down at the formula and quickly parsing the terms.} Let's go through the cases. We assume that $p_B$ and $p_M$ are not both zero; if they were, then neither can hit a basket. Not only would this mean that our ranking of them as two of the all-time greats is wrong, but the game will never end and thus there's no need to do any analysis!

\begin{enumerate}
\item Bird wins on his 1\textsuperscript{st} shot with probability $p_B$.
\item Bird wins on his 2\textsuperscript{nd} shot with probability $(1-p_B)(1-p_B)p_B = rp_B$.
\item Bird wins on his $n$\textsuperscript{th} shot with probability $(1-p_B)(1-p_M) \cdot (1-p_B)(1-p_M)$ $\cdots$ $(1-p_B)(1-p_M)p_B$ $=$ $r^{n-1}p_B$.
\end{enumerate}

To see this, if we want Bird to win on shot $n$ then we need to have him and Magic miss their first $n-1$ shots, which happens with probability $\left((1-p_B)(1-p_M)\right)^{n-1} = r^{n-1}$, and then Bird hits his $n$\textsuperscript{th} shot, which happens with probability $p_B$. The important thing to remember here is that we have broken the problem down into all of the possible ways Bird can beat Magic. In doing so, notice how the geometric series is surfacing! This makes sense since we have $n-1$ trials where Bird and Magic miss, and on the $n$\textsuperscript{th} shot, Bird makes the basket and wins the game. Thus
\begin{align*}
\text{Prob}(\text{Bird wins}) \ = \ x_B \ = \  p_B + rp_B + rp_B^2 +rp_B^3 + \cdots \ = \ p_B\sum_{n=0}^\infty r^n.
\end{align*}
which is a geometric series. As we assumed $p_B$ and $p_M$ are not both zero,  $r = (1-p_B)(1-p_M)$ satisfies $|r| < 1$ and we can use the geometric series formula to deduce \begin{align*} x_B \ = \ \frac{p_B}{1-r} \ = \ \frac{p_B}{1 - (1-p_B)(1-p_M)}.\end{align*}

We have made enormous progress. We converted our infinite series into a \textbf{\textit{closed-form expression}}, and we can easily see how the probability of Bird winning changes as we change $p_B$ and $p_M$; we display this in Figure \ref{fig:birdmagicplot}.

\begin{figure}
\begin{center}
\scalebox{1}{\includegraphics{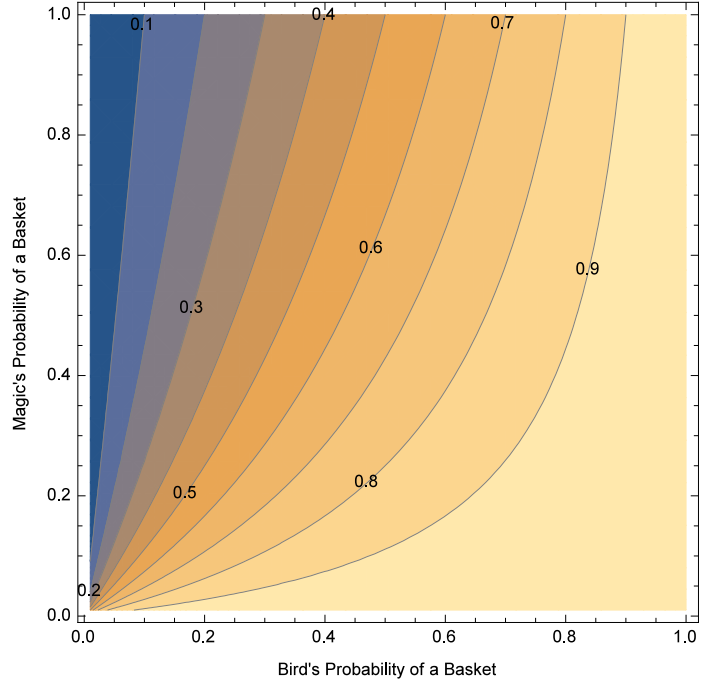}}
\caption{\label{fig:birdmagicplot} Probability Bird, shooting first, gets a basket before Magic.}
\end{center}\end{figure}


Note the plot supports our intuition. As the probability of Bird making a basket rises to 1, it doesn't matter what Magic's probability is as Bird will almost surely win on his first shot. Further, if the two probabilities are equal then Bird should win more than half of the time, as there is an advantage in going first.

\subsection{Solution through Memoryless Process and the Geometric Series Formula}

We now give a second solution to the basketball game. Not only does this approach avoid needing to know the geometric series formula, but it gives a proof of it!

Recall the assumptions we made. The probability Bird makes a shot is $p_B$, the probability Magic hits a basket is $p_M$, and the probability they both miss is $r := (1-p_B)(1-p_M)$. There is a lot hidden in these statements. We are assuming the two never tire; they always make baskets with a fixed probability. We can use this to compute $x_B$, the probability Bird wins, in another way. Before we wrote $x_B$ as a sum over the probabilities that Bird won in $n$ games. We claim that $$\text{Prob}(\text{Bird wins}) \ = \  x_B \ = \ p_B + rx_B.$$

To see this, note either Bird makes his first basket and wins (which happens with probability $p_B$) or he misses (with probability $1-p_B$). If Bird is going to win, then Magic must miss his first shot, and this happens with probability $1-p_M$. Something interesting happens, however, if both Bird and Magic miss: \emph{we have reset our game to its initial state!} Since both have missed, it's as if we just started playing the game right now. Note this would not be true if we stopped the analysis after Bird misses, as then Magic would have the next shot and the advantage. Since both miss and Bird has the ball again, by definition the probability Bird wins from this configuration is $x_B$, and thus the probability he wins is $p_B + (1-p_B)(1-p_M)x_B$.

Solving for $x_B$, the probability Bird beats Magic is \[x_B \ = \ \frac{p_B}{1-r_B}.\] As this must equal the infinite series expansion from the previous subsection, we deduce the geometric series formula: $$\frac{p_B}{1- r} \ = \ p_B \sum_{n=0}^\infty r^n \ \ \ {\rm therefore}\ \ \ \sum_{n=0}^\infty r^n \ = \ \frac1{1-r}.$$

\begin{rek} We have to be a bit careful. It's important to keep track of assumptions. In our analysis $r = (1-p_B)(1-p_M)$ with $0 \le p_B, p_M \le 1$ and both $p_B$ and $p_M$ are not zero. Thus we have only proved the geometric series formula if $0 \le r < 1$ (actually, if $p_B = 0$ we cannot divide both sides by $p_B$, and some care is needed). With a bit more work we can convert this to a proof for all real $|r| < 1$ by splitting the infinite sum into a sum over even and odd powers, and using the formula twice; we encourage you to make this rigorous. \end{rek}

Let's look closely at what we've done in this subsection. The key observation was to notice that we have a \textbf{\textit{memoryless process}}. In the infinite sum approach, which led to an infinite geometric series, we cared about each time Bird and Magic miss a free throw. In our new approach we just care about them both missing once. The reason is that if Bird and Magic both miss, the game essentially starts over, and the game has \emph{no memory} of what previously occurred. The  advantage to this method is that by reducing the game to the same state we start with, we turn an \emph{infinite} calculation into a \emph{finite} one! In general it is incredibly difficult to come up with a workable expression for an infinite series, and finite expressions are easier to compute. Thus, perhaps there is hope that we can convert the solution to the \mandm\ Game, equation \eqref{eq:probtiek}, into an equivalent finite sum....

\subsection{Lessons}

Before returning to the M\&M game, there are a few takeaways worthy of emphasis, all of which will resurface moving forward.

\begin{enumerate}
\item \textit{The Power of Perspective:} In the hoops game, the infinite series may have been daunting. However, after looking at the problem with a different perspective, we saw that we can use a memoryless process to attack an otherwise difficult problem. In fact, the memoryless process is one of the most powerful probability tools we have because it replaces the daunting challenge of infinite calculations with finite ones. Again, any time we can reduce an infinite problem to a finite problem is cause for celebration, as we are making enormous progress! (Technically infinite progress!)\\ \

\item \textit{Circumvent Algebra with Deeper Understanding:} Frequently there is a lot of messy algebra that goes into finding a formula for an infinite sum. The tricks we used to circumnavigate this algebra are great, and we should look for those types of shortcuts as often as possible.\\ \

\item \textit{The Depth of a Problem Is Not Always What You Expect:} Originally, we may have thought we needed the geometric series to solve this problem. It turns out we didn't! This will be a valuable insight for the M\&M game in that we should look for ways to simplify problems from infinite sums to finite expressions. That way, we don't have to deal with difficult infinite expressions.\\ \

\item \textit{Math is Fun:} How could anyone think otherwise?
\end{enumerate}

\section{Memoryless \mandms}

\subsection{Setup}

Remember (equation \eqref{eq:probtiek}) that we have an infinite sum for the probability of a tie with both people starting with $k$ \mandms:
\[ P(k,k) \ = \  \sum_{n=k}^{\infty} {n - 1 \choose k-1} \left(\frac{1}{2}\right)^{n-1}\frac{1}{2} \cdot {n - 1 \choose k-1} \left(\frac{1}{2}\right)^{n-1}\frac{1}{2}.\]


It's hard to evaluate this series as we have an infinite sum \textit{and} a squared binomial coefficient whose top is changing. Thus instead of evaluating this sum, which is very difficult to do, we want to somehow convert it to something where we have more familiarity. From the hoops game, we should be thinking about how to obtain a \textit{finite} calculation. The trick there was to notice we had a memoryless process, and all that mattered was the game state, not how we reached it. For our problem, we'll have many tosses of the coins, but in the end what matters is where we are, not the string of heads and tails that got us there.
	
Let's figure out some way to do this by letting $k=1$. In this case, we can do the same thing we did in the hoops game and boil the problem down into cases. There are four equally likely scenarios each time we toss coins, so the probability of each event occurring is 1/4 or 25\%.

\begin{enumerate}
\item Both players eat.
\item Cam eats an M\&M but Kayla does not.
\item Kayla eats an M\&M but Cam does not.
\item Neither eat.
\end{enumerate}	

These four possibilities lead to the infinite series in \eqref{eq:probtiek}, as we calculate the probability the game ends in $n$ tosses. It turns out one of the four events is not needed, and if we remove it we can convert to a finite game.

Let's revisit the lessons of the hoops game. There, we saw that we could create a \textbf{\textit{memoryless process}} by saying if Bird and Magic both missed their free throws, it was as if the game started over. We can do the same thing here: if Cam and Kayla both get tails and therefore don't eat their M\&Ms, then it's as if the coin toss never happened. We can therefore ignore the fourth possibility. If you want, another way to look at this is that if we toss two tails then there is no change in the number of \mandms\ for either kid, and thus we may pretend such a toss never happened. This allows us to remove all the tosses of double tails, and now after each toss at least one player, possibly both, have fewer \mandms. As we start with a finite number of \mandms, the game terminates in a finite number of moves.

Thus instead of viewing our game as having four alternatives each toss, there are only three and they all happen with probability 1/3. To see this, note that if ${\rm Pr}(X)$ is the probability that event $X$ happens, we now have a conditional probability problem\footnote{The standard notation is to write ${\rm Pr}(A|B)$ for the probability that $A$ happens, given that $B$ happens.}; for example, what is the probability Cam and Kayla both eat an \mandm\, \emph{given that the outcome is not double tails}? If $C$ denotes the event that Cam gets a head and eats (and $C^c$ the event that he gets a tail), and similarly $K$ for Kayla, then \begin{eqnarray} {\rm Pr}({\rm both\ eat}|{\rm at\ least\ one\ eats}) & \ = \ & \frac{{\rm Pr}(C \cap K)}{{\rm Pr}(C \cap K) + {\rm Pr}(C \cap K^c) + {\rm Pr}(C^c \cap K)} \nonumber\\ & \ = \ & \frac{1/4}{1/4 + 1/4 + 1/4} \ = \ \frac13. \nonumber \end{eqnarray}

We may therefore consider the related game with just three outcomes for each set of tosses, each happening with probability 1/3:

\begin{enumerate}
\item both players eat;
\item Cam eats an M\&M but Kayla does not;
\item Kayla eats an M\&M but Cam does not.
\end{enumerate}

Notice that after each toss the number of \mandms\ is decreased by either 1 or 2, so the game ends after at most $2k-1$ tosses.


\subsection{Solution}

Armed with the reduction from the previous subsection, we can replace the infinite sum of \eqref{eq:probtiek} with a finite sum.

\begin{thm}\label{thm:memorylessmandmfinitesum} The probability the \mandm\ Game ends in a tie with two people using fair coins and starting with $k$ \mandms\ is \be\label{eq:finitememoryless} P(k,k) \ = \  \sum_{n=0}^{k-1} {2k - n - 2 \choose n} \left(\frac{1}{3}\right)^n {2k - 2n - 2 \choose k - n - 1} \left(\frac{1}{3}\right)^{k-n-1} \left(\frac{1}{3}\right)^{k-n-1} \frac{1}{3}.\ee
\end{thm}

\begin{proof} Each of our three possibilities (both eat, just Cam eats, just Kayla eats) happens with probability 1/3. Since the game ends in a tie, we know the final toss must be double heads with both eating, and each must eat exactly $k-1$ \mandms\ in the earlier tosses. Let $n$ denote the number of times both eat before the final toss (which again we know must be double heads); clearly $n \in \{0, 1, \dots, k-1\}$. We thus have $n+1$ double heads, and thus Cam and Kayla must each eat $k-(n+1) = k - n - 1$ times when the other doesn't eat.

We see that, in the case where there are $n+1$ double heads (with the last toss being double heads), the total number of tosses is $$(n+1) + (k - n - 1) + (k - n - 1) \ = \ 2k - n - 1.$$ In the first $2k - n - 2$ tosses we must choose $n$ to be double heads, then of the remaining $(2k-n-2) - n = 2k - 2n - 2)$ tosses before the final toss we must choose $k-n-1$ to be just heads for Cam, and then the remaining $k-n-1$ tosses before the final toss must all be just heads for Kayla. These choices explain the presence of the two binomial factors. As each toss happens with probability 1/3, this explains those factors; note we could have just written $(1/3)^{2k-n-1}$, but we prefer to highlight the sources.
\end{proof}


\section{Viewing Data}


\subsection{Plotting Exact Answer}

Before turning to additional ways to solve the problem, it is worthwhile to pause for a bit and discuss how to view data and use results for small $k$ to predict results for larger ones.

While it is not obvious how we could replace the sum in \eqref{eq:finitememoryless} with a nice closed form expression involving $k$, this finite sum is certainly easier to use than the infinite sum in \eqref{eq:probtiek}. In fact, it's  very easy to use the finite sum to compute the exact answer. Below is some simple code to do so in Mathematica (and plot the result).\\ \

\begin{verbatim}
p[k_] := Sum[Binomial[2 k - n - 2, n] Binomial[2 k - 2 n - 2,
    k - n - 1] (1/3)^(2 k - n - 1), {n, 0, k - 1}]
tielist = {};
For[k = 1, k <= 1000, k++, tielist = AppendTo[tielist, {k, p[k]}]]
ListPlot[tielist, AxesLabel -> {"k", "Probability of a tie"}]
\end{verbatim}

\ \\

For example, if $k=1$ the probability of a tie is 1/3; this is quite reasonable, as there are three equally likely possibilities now and only one of them leads to a tie when both start with one \mandm. Some other fun values: if $k=2$ the probability is $5/27 \approx .185$, if $k=5$ it is $1921/19683 \approx .098$, if $k=10$ it falls to almost $.066$, while for $k=100$ it's about $.020$. See Figure \ref{fig:probtiekatmost1000} for more values.

\begin{figure}[h]
\begin{center}
\scalebox{1}{\includegraphics{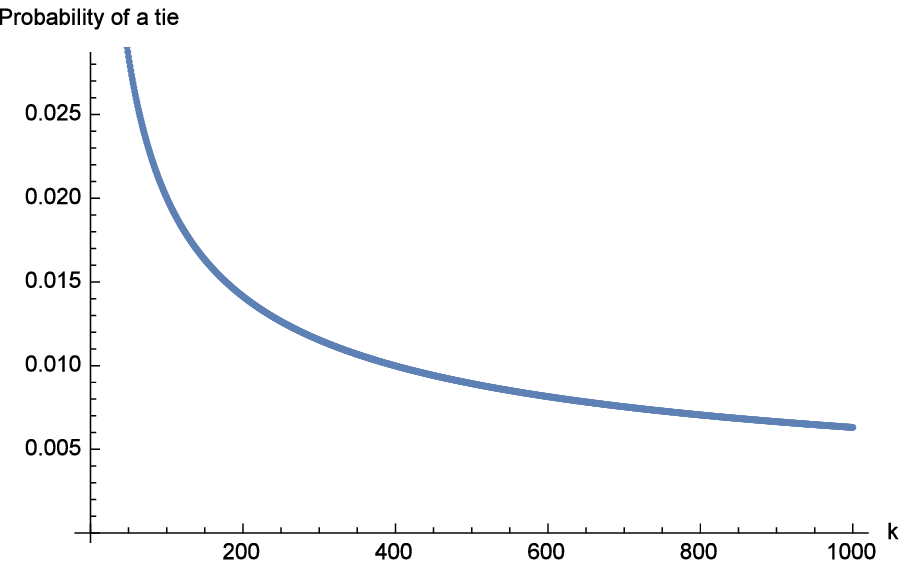}}
\caption{\label{fig:probtiekatmost1000} The probability of a tie for $k \le 1000$.}
\end{center}\end{figure}


\subsection{Log-log Plots}

While equation \eqref{eq:finitememoryless} gives us a nice formula for finite computations, it is hard to see the $k$ dependence. To try and guess how the answer varies with $k$ we can do a plot, but it's  hard to look at the results in Figure \ref{fig:probtiekatmost1000} and extrapolate to larger values of $k$. For example, what would you guess for the probability of a tie if there are 200 \mandms? If there are 2016?

An important skill to learn is how to view data. Frequently rather than plotting the data as given it's better to do a log-log plot. What this means is that instead of plotting the probability of a tie as a function of $k$, we plot the logarithm of the probability of a tie against the logarithm of $k$. We do this in Figure \ref{fig:logprobtiekatmost1000}.

\begin{figure}[h]
\begin{center}
\scalebox{1}{\includegraphics{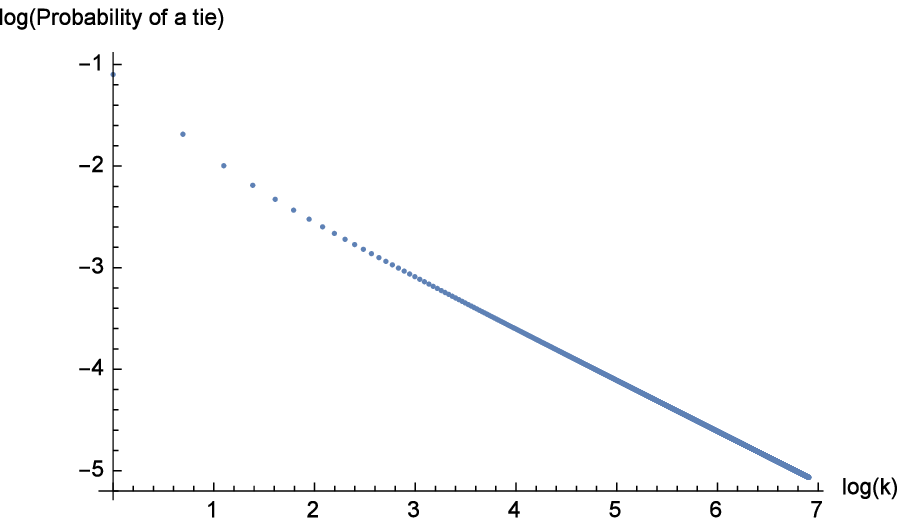}}
\caption{\label{fig:logprobtiekatmost1000} The probability of a tie for $k \le 1000$.}
\end{center}\end{figure}


Notice that the plot here looks \emph{very} linear. Lines are probably the easiest functions to extrapolate, and if this linear relationship holds we should be able to come up with a very good prediction for the logarithm of the probability (and hence by exponentiating obtain the probability). We do this in the next section.


\subsection{Statistical Inference}

Let's try to predict the answer for large values of $k$ from smaller ones. The fifth named author gave a talk on this at the 110\textsuperscript{th} meeting of the Association of Teachers of Mathematics in Massachusetts in March 2013, which will explain the prevalence of 110 and 220 below.

\begin{figure}[h]
\begin{center}
\scalebox{1}{\includegraphics{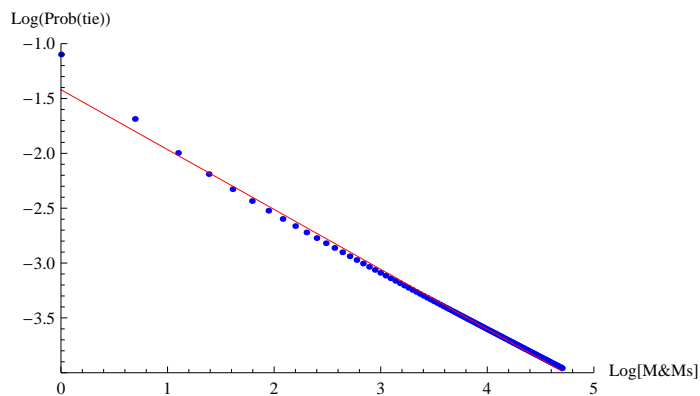}}
\caption{\label{fig:MandMgameProbTie1to110LogLog} The probability of a tie for $k \le 110$. The best fit line is good, but is noticeable non-perfect.}
\end{center}\end{figure}

Figure \ref{fig:MandMgameProbTie1to110LogLog} gives the log-log plot for $k \le 110$. Using the Method of Least Squares from Statistics\footnote{These formulas can be derived using multivariable calculus and linear algebra. For a derivation, see for example \bburl{http://web.williams.edu/Mathematics/sjmiller/public_html/105Sp10/handouts/MethodLeastSquares.pdf}.} with $P(k)$ the probability of a tie when we start with $k$ \mandms, we find a predicted best fit line of $$\log\left(P(k))\right) \ \approx \ -1.42022 - 0.545568 \log k,$$ or exponentiating $$P(k) \ \approx\ 0.2412 / k^{.5456}.$$ This predicts a probability of a tie when $k=220$ of about 0.01274, but the answer is approximately 0.0137. While we are close, we are off by a significant amount. (In situations like this it is better to look at not the difference in probabilities, which is small, but the percentage we are off; here we differ by about 10\%.)

Why are we so far off? The reason is that small values of $k$ are affecting our prediction more than the should. If we have a main term in the log-log plot which is linear, it will eventually dominate lower order terms \emph{but} those lower order terms could have a sizable effect for low $k$. Thus, it's a good idea to ignore the smaller values when extrapolating our best fit line.

In Figure \ref{fig:MandMgameProbTie50to110LogLog} we now go from $k=50$ to $110$.

\begin{figure}[h]
\begin{center}
\scalebox{1}{\includegraphics{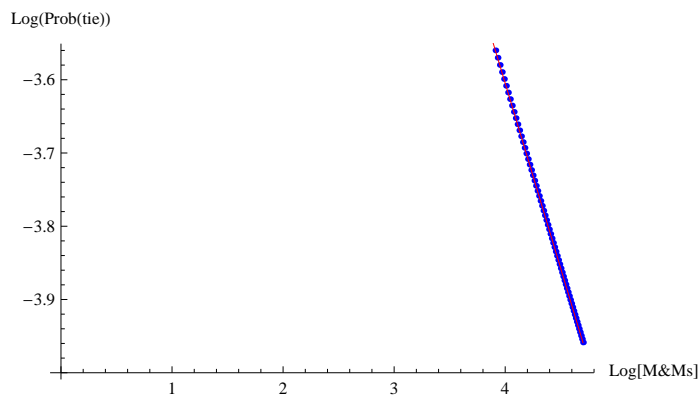}}
\caption{\label{fig:MandMgameProbTie50to110LogLog} The probability of a tie for $50 \le k \le 110$. The best fit line is almost indistinguishable from the data.}
\end{center}\end{figure}

Our new best fit line is  $$\log\left(P(k)\right) \ \approx \ -1.58261 - 0.50553  \log k,$$ or exponentiating $$P(k)\ \approx\ 0.205437 / k^{.50553};$$ we should compare this to our previous prediction of $0.241662 / k^{.5456}$). Using our new formula we predict 0.01344 for $k=220$, which compares \emph{very} favorably to the true answer of 0.01347.

The point of this section is to give  you a brief introduction to the power of statistics and extrapolating, and give you a sense of some of the issues in their use. We are able to get a fairly reasonable prediction with very little work, and if we clean up the data a little we improve to a phenomenal agreement.

\section{Recurrences}

As our goal is not to write a book on this game but rather to show how it leads to many good mathematical items, in the interest of space we will just briefly discuss two final approaches, recurrence relations in this section and hypergeometric functions in the next. See \cite{Mil} for a quick introduction to recurrences.

\subsection{Recurrence Review}

If you've seen the Fibonacci numbers $$\{F_n\}_{n=0}^\infty\ =\ \{0, 1, 1, 2, 3, 5, 8, \dots\},$$ you've seen a terrific example of a linear recurrence with constant coefficients, as they are the unique sequence satisfying $$F_{n+2} \ = \ F_{n+1} + F_n, \ \ \ F_0 \ = \ 0, \ \ \ F_1 \ = \ 1.$$ Once we know the relation and the first two coefficients, we can find any subsequent value by substituting. Unfortunately, this is costly in practice, as computing the $n$\textsuperscript{th} term requires us to know all the previous ones.

Fortunately, there are many ways to efficiently solve problems like this, and often these lead to beautiful closed form expressions. To solve the Fibonacci relation we guess $F_n = r^n$. Why is this reasonable? Clearly the Fibonacci sequence is non-decreasing, so $F_{n+2} \le 2 F_{n+1}$, which means that every time we increase the index by 1 we at most double our number, so $F_n \le 2^n$. Similarly we find $F_{n+2} \ge 2 F_n$; now increasing the index by 2 causes us to at least double, so increasing the index by 1 should yield an increase of at least a  factor of $\sqrt{2}$. Thus we expect the Fibonaccis to satisfy a relation such as $$2^{n/2}\ \le\ F_n\ \le\ 2^n,$$ which is highly suggestive of exponential growth; this is why we try $F_n = r^n$.

Substituting this into the recurrence we obtain the characteristic polynomial for $r$, which, after dividing both sides by $r^n$, is $$r^2 \ = \ r + 1;$$ the solutions to this are $$r_1 \ = \ \frac{1+\sqrt{5}}{2}, \ \ \ r_2 \ = \ \frac{1-\sqrt{5}}{2}.$$ A beautiful property of linear recurrences is that an arbitrary linear combination of solutions is a solution, and we find the general solution of the Fibonacci recurrence is $$F_n \ = \ c_1 r_!^n + c_2 r_2^n.$$ As we require $F_0 = 0$ and $F_1 = 1$, after some more algebra we obtain Binet's Formula, the spectacular relation $$F_n \ = \ \frac1{\sqrt{5}} \left(\frac{1+\sqrt{5}}{2}\right)^n - \frac1{\sqrt{5}} \left(\frac{1-\sqrt{5}}{2}\right)^n.$$ This formula is amazing: it gives us a simple, closed form expression for the $n$\textsuperscript{th} Fibonacci number; we can jump to this term in the sequence \emph{without} computing any of the earlier ones!\footnote{Additionally, as the Fibonaccis are all integers Binet's formula must return an integer; at first this might seem unlikely, as our expression involves square-roots and fractions, but fortunately everything that needs to cancel does.}

The point of the above is to give a brief glimpse at the rich theory; there is far more that could be said (especially concerning generating function approaches to solve problems such as these), but for our purposes this suffices. The main takeaways is that sometimes we are lucky and able to derive simple closed form expressions, but even if we cannot we are often able to determine the terms by repeated application of the defining relation and initial conditions.

\subsection{The \mandm\ Recurrence}

Even though we have a finite sum for the probability of a tie (equation  \ref{eq:finitememoryless}), finding that required some knowledge of combinatorics and binomial coefficients. We give an alternate approach which avoids these ideas. It's possible to do it with or without noting that we have a memoryless process. We'll do the memoryless process first as we'll assume we're still clever enough to notice that, and then remark afterwards how we would have found the same formula even if we didn't realize this.

We need to consider a more general problem. We always denote the number of \mandms\ Cam has with $c$, and Kayla with $k$\footnote{We can see the power of good notation. Using $c$ and $k$ to represent the number of \mandms\  Cam and Kayla each have allows us to know exactly what is going on as the math gets more involved.}; we frequently denote this state by $(c,k)$. Then we can rewrite the three equally likely scenarios, each with probability 1/3, as follows:\\ \

\begin{itemize}
\item $(c,k) \longrightarrow (c-1, k-1)$ (double heads and both eat),
\item $(c,k) \longrightarrow (c-1, k)$ (Cam gets a head and Kayla a tail),
\item $(c,k) \longrightarrow (c, k-1)$ (Cam gets a tail and Kayla a head).\\ \
\end{itemize}

If we let $x_{c,k}$ denote the probability the game ends in a tie when we start with Cam having $c$ \mandms\ and Kayla having $k$, we can use the above to set up a recurrence relation. How so? Effectively, on each turn we move from $(c,k)$ in exactly one of the following three ways: either Cam and Kayla both eat an M\&M in which case $(c,k) \longrightarrow (c-1, k-1)$; only Cam flips heads and eats an M\&M in which case $(c,k) \longrightarrow (c-1, k)$; or Kayla is the only one to eat an M\&M in which case $(c,k) \longrightarrow (c, k-1)$. Now, we can use simpler game states to figure out how the probability of a tie when we start with more M\&M, as in each of the three cases we have reduced the total number of \mandms\ by at least one. We thus find that the recurrence relation satisfied by $\{x_{c,k}\}$ is
\begin{equation}\label{eq:mandmrecurrence}
x_{c,k} \ = \  \frac{1}{3} x_{c-1,k-1} + \frac{1}{3} x_{c-1,k} + \frac{1}{3} x_{c,k-1} \ = \  \frac{x_{c-1,k-1} + x_{c-1,k} + x_{c,k-1}}{3}.
\end{equation}

From our work on the Fibonacci numbers we know that cannot be the full story -- we need to specify initial conditions. A little thought says $x_{0,0}$ must be 1 (if they both have no \mandms\ then it must be a tie), while $x_{c,0} = 0$ if $c>0$ and similarly $x_{0,k} = 0$ if $k > 0$ (as in these cases exactly one of them has an \mandm, and thus the game cannot end in a tie).

We have made tremendous progress. We use these initial values and the recurrence relation \eqref{eq:mandmrecurrence} to determine $x_{c,k}$. Unfortunately we cannot get a simple closed form expression, but we can easily compute the values by recursion. A good approach is to compute all $x_{c,k}$ where $c+k$ equals sum fixed sum $s$. We've already done the cases $s = 0$ and $s=1$, finding $x_{0,0} = 1$, $x_{0,1} = x_{1,0} = 0$.

We now move to $s=2$. We need only find $x_{1,1}$, as we know $x_{2,0} = x_{0,2} = 0$. Using the recurrence relation we find
$$x_{1,1} \ = \  \frac{x_{0,0} + x_{0,1} + x_{1,0}}{3} \ = \  \frac{1 + 0 + 0}{3} \ = \  \frac{1}{3}.$$

Next is the case when the indices sum to 3. Of course, $x_{0,3} = x_{3,0} = 0$, so all we need are $x_{1,2}$ and $x_{2,1}$ (which by symmetry are the same). We find
\[x_{2,1} \ = \  x_{1,2} \ = \  \frac{x_{1,1} + x_{2,0} + x_{0,2}}{3} \ = \  \frac{1/3 + 0 + 0}{3} \ = \  \frac{1}{9}.\]

We can continue to $s=4$, and after some algebra easily obtain \[x_{2,2} \ = \  \frac{x_{1,1} + x_{2,1} + x_{1,2}}{3} \ = \  \frac{5}{27}.\]

If we continued on with these calculations, we would find that $x_{3,3} = \frac{11}{81}$, $x_{4,4} = \frac{245}{2187}$, $x_{5,5} = \frac{1921}{19863}$, $x_{575,6561} = \frac{11}{81}$, $x_{7,7} = \frac{42635}{531441}$, and $x_{8,8} = \frac{355975}{4782969}$.  The beauty of this recursion process is that we have a sure-fire way to figure out the probability of a tie at different states of the M\&M game. We leave it as an exercise to the interested reader to compare the computational difficulty of finding $x_{100,100}$ by the recurrence relation versus by the finite sum \eqref{eq:finitememoryless}.

We end with one final comment on this approach. It's possible to recast this problem as one in counting paths on a graph. In Figure \ref{fig:MandMgame} we start with $(c,k) = (4,4)$, and look at all the possible paths that end in $(0,0)$. The probability of any path is equal to $(1/3)^t$, where $t$ is the number of terms in the path. It turns out the solution is very similar to the famous Catalan numbers, which count the number of paths from $(0,0)$ to $(n,n)$ moving in unit horizontal or unit vertical steps and never going above the main diagonal; the difference here is that we now have three possible choices at each turn.

\begin{figure}[h]
\begin{center}
\scalebox{.7}{\includegraphics{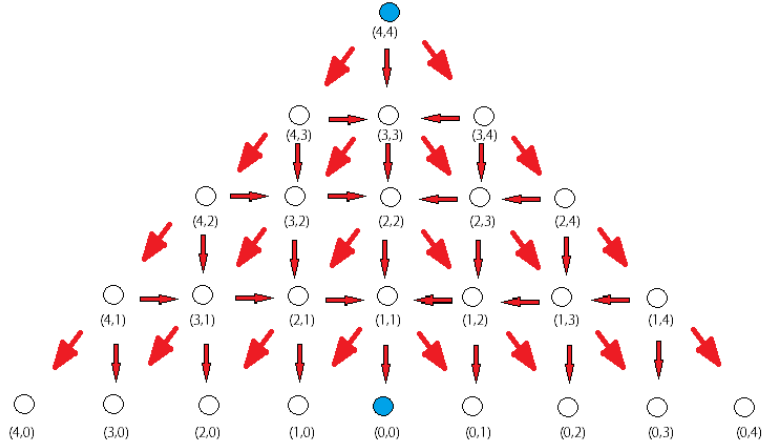}}
\caption{\label{fig:MandMgame} The M\&M game when $k=4$. Count the paths! Answer 1/3 of probability hit (1,1).}
\end{center}\end{figure}

\subsection{Forgetting Memoryless Processes}

In the previous subsection we found a recurrence relation for $x_{c,k}$, but our analysis was based on there only being three options at each step. What if we hadn't noticed there was a memoryless process lurking: would we still have found the same relation? In that case, there would now be four possibilities on each turn, each happening with probability 1/4.\\ \

\begin{itemize}
\item $(c,k) \longrightarrow (c-1, k-1)$ (double heads and both eat),
\item $(c,k) \longrightarrow (c-1, k)$ (Cam gets a head and Kayla a tail),
\item $(c,k) \longrightarrow (c, k-1)$ (Cam gets a tail and Kayla a head),
\item $(c,k) \longrightarrow (c, k)$ (double tails and neither eats). \\ \
\end{itemize}

We now obtain the following relation: $$x_{c,k} \ = \ \frac14 x_{c-1,k-1} + \frac14 x_{c-1,k} + \frac14 x_{c,k-1} + \frac14 x_{c,k}.$$ Note that if we bring the $\frac14 x_{c,k}$ over to the left hand side we relate $\frac34 x_{c,k}$ to multiples of $x_{c-1,k-1}, x_{c-1,k}$ and $x_{c,k-1}$: $$\frac34 x_{c,k} \ = \ \frac14 x_{c-1,k-1} + \frac14 x_{c-1,k} + \frac14 x_{c,k-1}.$$ If we then multiply through by $4/3$ we regain our old recurrence, equation \eqref{eq:mandmrecurrence}: $$x_{c,k} \ = \ \frac13 x_{c-1,k-1} + \frac13 x_{c-1,k} + \frac13 x_{c,k-1}.$$

This is wonderful: it means if we did not initially notice that there was a memoryless process, doing the algebra suggests there should be one!


\subsection{Revisiting and Generalizing the Hoops Game}

When you learn new concepts in math, it often pays great dividends to revisit earlier problems. Let's go back to the hoops game; not surprisingly, we'll see that, similar to the \mandm\ Game, we can cast it as a double recurrence.

The way the recurrence method worked was we reduced the problem we wished to study to a simpler state; however, that required us to know those answers. Thus it is not enough to deal with just $x_B$ in general, and we need to study $x_{B;b,m}$, which is the probability Bird wins when he needs to make $b$ more baskets to win, Magic needs to make $m$ more to win, \emph{and  Bird is currently shooting!}.\footnote{Another option is that we could introduce $x_{M;b,m}$, the corresponding probability where now Magic is shooting.} We find
\bea\label{eq:basketballrecurrence} x_{B;b,m} & \ =\ & p_B p_M x_{B;b-1,m-1}  + p_B (1-p_M) x_{B;b-1,m} \nonumber\\ & & \ \ \  +\ (1 - p_B) p_M x_{B;b,m-1} + (1 - p_B) (1 - p_M) x_{B;b,m}. \eea To see why this is true, let's look at the first term. The $p_B$ means Bird got a basket, reducing the number he needs by 1. Note that if $b-1 = 0$ then Bird wins and the game should stop (we'll deal more with this in a moment). Now it's Magic's turn to shoot. If he gets a basket, which happens with probability $p_M$, that reduces his number of baskets needed to $m-1$, which explains the $p_B p_M x_{B;b-1,m-1}$ term; the other three terms arise from the other possibilities.

We also need the initial conditions. Clearly $x_{B;b,0} = 0$ if $b > 0$ and $x_{B;0,m} = 1$ if $m > 0$, but what should we choose for $x_{B;0,0}$? Well, the way to interpret this is that each needs to make zero baskets and Bird shoots first, so he is the first to reach zero baskets. Thus we set $x_{B;0,0}$ equal to 1. Another justification is that we only reach this situation when Bird makes a basket and then Magic, who shouldn't have been allowed to shoot as Bird just won the game, shoots. These normalizations often are tricky, but can frequently be determined by a good story. This is similar to the convention that $0! = 1$, which we interpret as there is only one way to do nothing (i.e.,  there is one way to order no elements -- there shouldn't be multiple ways to do nothing!).

We have thus found a recurrence for $x_{B;b,m}$. Let's check and make sure it reduces to our previous result when $b=m=1$. In that case, equation \eqref{eq:basketballrecurrence}  becomes $$x_{B;1,1} \  = \ p_B p_M 1 + p_B (1-p_M) 1 + (1-p_B) p_M 0 + (1 - p_B) (1 - p_M) x_{B;1,1}.$$ Remembering that we use $r$ for $(1 - p_B) (1 - p_M)$ after some simple algebra we obtain $$x_{B;1,1} \ = \ \frac{p_B}{1 - r},$$ exactly as before! Of course, this now suggests a natural question: what does $x_{B;b,b}$ look like as $b$ grows (let's say under the assumption that $p_B = p_M$)?

\section{Hypergeometric Functions}

We end our tour of solution approaches with a method that actually prefers the infinite sum to the finite one, hypergeometric functions (see for example \cite{AS, GR}). These functions arise as the solution of a particular linear second order differential equation:
\[x(1-x)y''(x) + [c - (1-a+b)x]y'(x) - a b y(x) \ = \ 0\] (this is also called Gauss's differential equation). This equation is useful because every other linear second order differential equation with three singular points (in the case they are at 0, 1, and $\infty$) can be transformed into it. As this is a second order differential equation there should be two solutions. One is
\[y(x)\ = \ 1 + \frac{abx}{c1!}+\frac{a(a+1)b(b+1)x^2}{c(c+1)2!} + \frac{a(a+1)(a+2)b(b+1)(b+2)x^3}{c(c+1)(c+2)3!} + \cdots,\] so long as
$c$ is not a non-positive integer; we denote this solution by ${\ }_{2}F_{1}(a,b;c;z)$. By choosing appropriate values of $a, b$  and $c$ we recover many special functions. Wikipedia lists three nice examples: $$\log(1+x) \ = \ x{\ }_{2}F_{1}(1,1;2;-x), \ \ \ (1-x)^{-a} \ = \ {\ }_{2}F_{1}(a,1;1;x), \ \ \ \arcsin(x) \ = \ x {\ }_{2}F_{1}(1/2,1/2;3/2;x^2).$$

By introducing some notation we can write the series expansion more concisely. We define the Pochhammer symbol by $$(a)_n\ =\ a (a+1) (a+2) \cdots (a+n-1) \ = \ \frac{(a+n-1)!}{(a-1)!}$$ (where the last equality holds for integer $a$; for real $a$ we need to interpret the factorial as its completion, the Gamma function). Our solution becomes
\[{\ }_{2}F_{1}(a,b,c;x) \ = \ \sum_{n=0}^\infty \frac{(a)_n (b)_n x^n}{(c)_n n!}.\] Note the factorials in the above expression suggest that there should be connections between hypergeometric functions and products of binomial coefficients. In this notation, the 2 represents the number of Pochhammer symbols in the numerator, the 1 the number of Pochhammer symbols in the denominator, and the $a$, $b$, and $c$ are what we evaluate the symbols at (the first two are the ones in the numerator, the last the denominator). One could of course consider more general functions, such as
\[{\ }_{s}F_{t}(\{a_i\},\{b_j\};x) \ = \ \sum_{n=0}^\infty \frac{(a_1)_n\cdots (a_s)_n x^n}{(b_1)_n \cdots (b_t)_n n!}.\]

The solution ${\ }_{2}F_{1}(a,b,c;x)$ is called a hypergeometric function, and if you look closely at it while recalling the infinite sum solution to the M\&M Game you might see the connection. After some algebra where we convert the binomial coefficients in the infinite sum solution \eqref{eq:probtiek} to the falling factorials that are the Pochhammer symbols, we find the following closed form solution.

\begin{thm} The probability the \mandm\ Game ends in a tie with two people using fair coins and starting with $k$ \mandms\ is \be\label{eq:probtiekhypergeometric} P(k,k) \ = \ {\ }_2F_{1}(k,k,1;1/4)4^{-k}.\ee \end{thm}

It is not immediately clear that this is progress; after all, it looks like we've just given a fancy name to our infinite sum. Fortunately, special values of hypergeometric functions are well studied, and a lot is known about their behavior as a function of their parameters. We encourage the interested reader to explore the literature and discover how `useful' the above is.


\section{OEIS}

We end with a short bonus section on how to guess formulas. There is an enormous wealth of information available on-line, but often it is  hard to figure out what we need and where it resides. A terrific resource is the On-Line Encyclopedia of Integer Sequences (OEIS, \bburl{http://oeis.org/}). This is a wonderful resource with a large number of integer sequences tabulated and stored. You enter some known terms in your sequence, and the site not only tells you what sequences it knows that agree with this, but it provides links, properties and formulas when it can!

For example, if we use our finite series expansion \eqref{eq:finitememoryless} or the recurrence relation \eqref{eq:mandmrecurrence} we can easily calculate the probability of a tie for some small $k$. We give the probabilities for $k$ up to 8 in Table \ref{table:probtie}. In addition, we also give $3^{2k-1} P(k,k)$. The reason we do this is that looking at the probability of a tie one is struck by the fact that the denominators are all powers of 3; after a little algebra we see that if we multiply by $3^{2k-1}$ we clear the denominators, and we will obtain a sequence of \emph{integers}. Note that it is very important that we end with integers and not rational numbers if we wish to use the OEIS.

\begin{center}
\begin{table}[h]
\begin{tabular}{|c||r|r|}
  $k$ & $P(k,k)$ & $3^{2k-1} P(k,k)$ \\
  \hline
  1 & 1/3 & 1 \\
  2 & 5/27 & 5 \\
  3 & 11/81 & 33  \\
  4 & 245/2187 & 245 \\
  5 &  1921/19683& 1921 \\
  6 & 575/6561 &  15525\\
  7 & 42635/531441 &  127905\\
  8 & 355975/4782969 & 1067925 \\
\end{tabular}
\caption{Probability of a tie as a function of the number of \mandms\ the two players have.}\label{table:probtie}
\end{table}
\end{center}

\begin{rek} If we didn't notice the right power of 3, we could have reached the same conclusion another way. There are three possibilities each time; from Figure \ref{fig:MandMgame} we saw our problem is equivalent to counting how  many paths there are from $(k,k)$ to $(0,0)$. As we end at $(0,0)$ our last step is forced, and the longest path comes when we never get double heads. As we must remove $2k-2$ \mandms\ (remember the last toss of double heads removed 2 of the $2k$), the longest path has $2k-2+1 = 2k-1$ steps, explaining the presence of this factor as the exponent of $3$. \end{rek}

Thus to the \mandm\ Game with two players we can associate the integer sequence 1, 5, 33, 245, 1921, 15525, 127905, 1067925, $\dots$. We plug that into the OEIS and find that it knows that sequence! It is sequence A084771 (see \bburl{http://oeis.org/A084771}). The very first comment there on this sequence is that it equals the number of paths in the graph we discussed!

The OEIS is a powerful tool for research. Think back to proofs by induction: if you are told what to prove, it is a lot easier and often the proof writes itself. The OEIS frequently gives you such an advantage.

\section{Takeaways and Further Questions}

We've seen many different ways of solving the \mandm\ Game, each leading to a different important aspect of mathematics. We end with some quick reflections on some  of the valuable lessons this game has to offer.\\ \
	
\begin{enumerate}
\item \textbf{Ask Questions:} Great mathematics is everywhere, waiting to be realized and explored. Often some of the deepest mathematics can be extracted from some of the most straightforward problems to state (for an excellent example see Conway's \emph{See and Say} (or \emph{Look and Say}) sequence. \\ \

\item \textbf{There are Many Ways to Solve Problems:} Different ways of solving the problem have different advantages. With the recurrence relation, we can calculate the answer for any given number of M\&Ms, but it might take a long time. With hypergeometric functions, we get a nice closed-form way of representing our solution. And with a memoryless process we get a finite sum involving well-understood binomial coefficients.\\ \

\item \textbf{Experience is a Useful and Great Guide:} We were able to make enormous leaps in the M\&M problem because we had the hoops game as a reference. So much of math is interconnected that a lot of times, all it takes for us to solve a difficult problem is remembering what we have done in the past.\\ \

\item \textbf{Need to Look at Data the Right Way:} When we looked at the numbers properly, we were able to make progress in solving the M\&M problem. This is no different for other math problems: so often, all it takes to figure out a complicated problem is the right lens.\\ \
\end{enumerate}

All of these takeaways are great, and they should lead to one last exercise: asking further questions. How long do we expect a game to take? What would happen to the M\&M problem if we increased the number of players? What about if all of the players started with different numbers of \mandms? Maybe the game would yield interesting results if the participants used biased coins.

In one of the first games ever played, Cameron, Kayla and Steven Miller each started with five \mandms\ and Kayla tossed five consecutive heads, dying immediately; years later she still talks about that memorable performance. There is a lot known about the longest run of heads or tails in tosses of a fair (or biased) coin (see for example \cite{Sch}). We can ask related questions here. What is the expected longest run of heads or tails by any player in a game? What is the expected longest run of tosses where all players' coins have the same outcome?

We could also revisit the hoops game and consider generalizations there. What if Bird and Magic keep shooting until someone makes $k$ baskets. What's the probability of a tie now? What if you keep shooting until you miss? We could also ask questions about streaks of hits and misses within the game. For another possibility, what if Bird's probability of making a basket and Magic's probability of a basket are independent random variables drawn uniformly\footnote{This means that for any interval $[a,b] \subset [0,1]$, the probability $p_B \in [a,b]$ is $b-a$, similarly for $p_M$, and knowledge of $p_B$  gives no information on $p_M$ (or vice-versa).} on $[0,1]$: what is the probability that Bird has a greater chance of winning than Magic? If you look at Figure \ref{fig:birdmagicplot} this problem with $k=1$ is equivalent to finding the area in the unit square above and to the left of the contour line with value .5; the answer turns out to be $\log(2) \approx .693147$. Is there a nice answer for general $k$? What if instead we ask what is the probability Bird wins if $p_B$ and $p_M$  are independently drawn uniform random variables on $[0, 1]$? If $k=1$ the answer is $\pi^2/6 - 1 \approx .644934$. It's nice that in both phrasings the answers are interesting numbers, and that the two different interpretations are quite close.

There are plenty of further questions out there, all of which would provide great insights not only into the M\&M game and its educational value but also into the study of math in general. We hope you will explore some of these or, even better, ones of your own choosing, and let us know what you find!


\ \\

\end{document}